\documentclass[a4paper]{amsart}

\usepackage{amscd}
\usepackage{amsmath}
\usepackage{amssymb}
\usepackage{amsthm}
\usepackage{bbm}
\usepackage{stmaryrd}

\usepackage[T1]{fontenc}
\usepackage[utf8]{inputenc}

\newtheorem{satz}{Satz}[section]
\newtheorem{theorem}[satz]{Theorem}
\newtheorem{proposition}[satz]{Proposition}

\newtheorem{lemma}[satz]{Lemma}

\newtheorem{remark}[satz]{Remark}

\begin{document}

\title[Compact embeddings for spaces of forward rate curves]{Compact embeddings for spaces of forward rate curves}
\author{Stefan Tappe}
\address{Leibniz Universit\"{a}t Hannover, Institut f\"{u}r Mathematische Stochastik, Welfengarten 1, 30167 Hannover, Germany}
\email{tappe@stochastik.uni-hannover.de}
\begin{abstract}
The goal of this note is to prove a compact embedding result for spaces of forward rate curves. As a consequence of this result, we show that any forward rate evolution can be approximated by a sequence of finite dimensional processes in the larger state space.
\end{abstract}
\maketitle

\section{Introduction}

The Heath-Jarrow-Morton-Musiela (HJMM) equation is a stochastic partial differential equation that models the evolution of forward rates in a market of zero coupon bonds; we refer to \cite{fillnm} for further details. It has been studied in a series of papers, see, e.g. \cite{Rusinek, Barski}, \cite{Filipovic-Tappe, Positivity} and references therein. The state space, which contains the forward curves, is a separable Hilbert space $H$ consisting of functions $h : \mathbb{R}_+ \rightarrow \mathbb{R}$. In practice, forward curves have the following features:
\begin{itemize}
\item The functions $h \in H$ become flat at the long end.

\item Consequently, the limit $\lim_{x \rightarrow \infty} h(x)$ exists.
\end{itemize}
The second property is taken into account by choosing the Hilbert space
\begin{align*}
L_{\beta}^2 \oplus \mathbb{R},
\end{align*}
where $L_{\beta}^2$ denotes the weighted Lebesgue space
\begin{align}\label{def-L-space}
L_{\beta}^2 := L^2(\mathbb{R}_+,e^{\beta x} dx)
\end{align}
for some constant $\beta > 0$.
Such spaces have been used, e.g., in \cite{Rusinek, Barski}.
As flatness of a function is measured by its derivative, the first property is taken into account by choosing the space
\begin{align}\label{def-H-space}
H_{\gamma} := \{ h : \mathbb{R}_+ \rightarrow \mathbb{R} : h \text{ is absolutely continuous with } \| h \|_{\gamma} < \infty \}
\end{align}
for some constant $\gamma > 0$, where the norm is given by
\begin{align}\label{norm-subspace}
\| h \|_{\gamma} := \bigg( |h(0)|^2 + \int_{\mathbb{R}_+} |h'(x)|^2
e^{\gamma x} dx \bigg)^{1/2}.
\end{align}
Such spaces have been introduced in \cite{fillnm} (even with more general weight functions) and further utilized, e.g., in \cite{Filipovic-Tappe, Positivity}.
Our goal of this note is to show that for all $\gamma > \beta > 0$ we have the compact embedding
\begin{align*}
H_{\gamma} \subset \subset L_{\beta}^2 \oplus \mathbb{R},
\end{align*}
that is, the forward curve spaces used in \cite{fillnm} and forthcoming papers are contained in the forward curve spaces used in \cite{Rusinek}, and the embedding is even compact. Consequently, the embedding operator between these spaces can be approximated by a sequence of finite-rank operators, and hence, when considering the HJMM equation in the state space $H_{\gamma}$, applying these operators its solutions can be approximated by a sequence of finite dimensional processes in the larger state space $L_{\beta}^2 \oplus \mathbb{R}$; we refer to Section~\ref{sec-main} for further details.

The remainder of this note is organized as follows. In Section~\ref{sec-pre} we provide the required preliminaries. In Section~\ref{sec-main} we present the embedding result and its proof, and we outline the described approximation result concerning solutions of the HJMM equation.

\section{Preliminaries and notation}\label{sec-pre}

In this section, we provide the required preliminary results and some basic notation. Concerning the upcoming results about Sobolev spaces and Fourier transforms, we refer to any textbook about functional analysis, such as \cite{Rudin} or \cite{Werner}.

As noted in the introduction, for positive real numbers $\beta,\gamma > 0$ the separable Hilbert spaces $L_{\beta}^2 \oplus \mathbb{R}$ and $H_{\gamma}$ are given by (\ref{def-L-space}) and (\ref{def-H-space}), respectively. These spaces and the forthcoming Sobolev spaces will be regarded as spaces of complex-valued functions. For every $h \in H_{\gamma}$ the limit $h(\infty) := \lim_{x \rightarrow \infty} h(x)$ exists and the subspace
\begin{align*}
H_{\gamma}^0 := \{ h \in H_{\gamma} : h(\infty) = 0 \}
\end{align*}
is a closed subspace of $H_{\gamma}$, see \cite{fillnm}.
For an open set $\Omega \subset \mathbb{R}$ we denote by $W^1(\Omega)$ the Sobolev space
\begin{align*}
W^1(\Omega) := \{ f \in L^2(\Omega) : f' \in L^2(\Omega) \text{ exists} \},
\end{align*}
which, equipped with the inner product
\begin{align}\label{norm-Sobolev}
\langle f,g \rangle_{W^1(\Omega)} = \langle f,g \rangle_{L^2(\Omega)} + \langle f',g' \rangle_{L^2(\Omega)},
\end{align}
is a separable Hilbert space. Here, derivatives are understood as weak derivatives.

For a function $h \in W^1((0,\infty))$ the extension $h \mathbbm{1}_{(0,\infty)} : \mathbb{R} \rightarrow \mathbb{C}$ does, in general, not belong to $W^1(\mathbb{R})$. In the present situation, this technical problem can be resolved as follows.
Let $h : (0,\infty) \rightarrow \mathbb{C}$ be a continuous function such that the limit $h(0) := \lim_{x \rightarrow 0} h(x)$ exists. Then we define the reflection $h^* : \mathbb{R} \rightarrow \mathbb{C}$ as
\begin{align*}
h^*(x) :=
\begin{cases}
h(x), & \text{if $x \geq 0$,}
\\ h(-x), & \text{if $x < 0$.} 
\end{cases}
\end{align*}

\begin{lemma}\label{lemma-reflection}
The following statements are true:
\begin{enumerate}
\item For each $h \in W^1((0,\infty))$ we have $h^* \in W^1(\mathbb{R})$.

\item The mapping $W^1((0,\infty)) \rightarrow  W^1(\mathbb{R})$, $h \mapsto h^*$ is a bounded linear operator.

\item For each $h \in  W^1((0,\infty))$ we have 
\begin{align*}
\| h \|_{W^1((0,\infty))} &\leq \| h^* \|_{W^1(\mathbb{R})} \leq \sqrt{2} \| h \|_{W^1((0,\infty))},
\\ \| h \|_{L^2((0,\infty))} &\leq \| h^* \|_{L^2(\mathbb{R})} \leq \sqrt{2} \| h \|_{L^2((0,\infty))}.
\end{align*}
\end{enumerate}
\end{lemma}

\begin{proof}
This follows from a straightforward calculation following the proof of \cite[Theorem~8.6]{Brezis}.
\end{proof}

\begin{lemma}\label{lemma-forward-Sobolev}
Let $\gamma > \beta > 0$ be arbitrary. Then the following statements are true:
\begin{enumerate}
\item We have $H_{\gamma}^0 \subset H_{\beta}^0$ and
\begin{align}\label{embedding-same}
\| h \|_{\beta} \leq \| h \|_{\gamma} \quad \text{for all $h \in H_{\gamma}^0$.}
\end{align}

\item We have $H_{\gamma}^0 \subset L_{\beta}^2$ and there is a constant $C_1 = C_1(\beta,\gamma) > 0$ such that
\begin{align}\label{embedding-diff}
\| h \|_{L_{\beta}^2} \leq C_1 \| h \|_{\gamma} \quad \text{for all $h \in H_{\gamma}^0$.}
\end{align}

\item For each $h \in H_{\gamma}^0$ we have 
\begin{align*}
h e^{(\beta / 2) \bullet}|_{(0,\infty)} \in W^1((0,\infty)), \quad (h e^{(\beta / 2) \bullet}|_{(0,\infty)})^* \in W^1(\mathbb{R}),
\end{align*}
and there is a constant $C_2 = C_2(\beta,\gamma) > 0$ such that
\begin{align*}
\| (h e^{(\beta / 2) \bullet} |_{(0,\infty)})^* \|_{W^1(\mathbb{R})} &\leq C_2 \| h \|_{\gamma} \quad \text{for all $h \in H_{\gamma}^0$.}
\end{align*}
\end{enumerate}
\end{lemma}

\begin{proof}
The first statement is a direct consequence of the representation of the norm on $H_{\gamma}^0$ given by (\ref{norm-subspace}). Let $h \in H_{\gamma}^0$ be arbitrary.
By the Cauchy-Schwarz inequality we obtain
\begin{align*}
\| h \|_{L_{\beta}^2}^2 &= \int_{\mathbb{R}_+} |h(x)|^2 e^{\beta x} dx =
\int_{\mathbb{R}_+} \bigg( \int_{x}^{\infty} h'(\eta)
e^{(\gamma/2) \eta} e^{-(\gamma/2) \eta} d \eta
\bigg)^2 e^{\beta x} dx
\\ &\leq \int_{\mathbb{R}_+} \bigg( \int_{x}^{\infty} |h'(\eta)|^2 e^{\gamma \eta} d\eta
\bigg) \bigg( \int_{x}^{\infty} e^{- \gamma \eta} d\eta \bigg)
e^{\beta x} dx 
\\ &\leq \int_{\mathbb{R}_+} \bigg( \int_{\mathbb{R}_+} |h'(\eta)|^2 e^{\gamma \eta} d\eta \bigg) \frac{1}{\gamma} e^{-\gamma x} e^{\beta x} dx
\\ &\leq \frac{1}{\gamma} \bigg( \int_{\mathbb{R}_+} e^{-(\gamma - \beta)x} dx \bigg) \| h \|_{\gamma}^2 = \frac{1}{\gamma(\gamma - \beta)} \| h \|_{\gamma}^2,
\end{align*}
proving the second statement. Furthermore, by (\ref{embedding-diff}) we have
\begin{align*}
\| h e^{(\beta / 2) \bullet} |_{(0,\infty)} \|_{L^2((0,\infty))}^2 
&= \int_{\mathbb{R}_+} |h(x) e^{(\beta / 2)x}|^2 dx = \int_{\mathbb{R}_+} |h(x)|^2 e^{\beta x} dx 
\\ &= \| h \|_{L_{\beta}^2}^2 \leq C_1^2 \| h \|_{\gamma}^2,
\end{align*}
and by estimates (\ref{embedding-same}), (\ref{embedding-diff}) we obtain
\begin{align*}
&\| (d/dx) (h e^{(\beta / 2) \bullet} |_{(0,\infty)} ) \|_{L^2((0,\infty))}^2 = \int_{\mathbb{R}_+} \Big| \frac{d}{dx} \big( h(x) e^{(\beta/2) x} \big) \Big|^2 dx 
\\ &= \int_{\mathbb{R}_+} \Big| h'(x) e^{(\beta/2) x} + \frac{\beta}{2} h(x) e^{(\beta/2) x} \Big|^2 dx
\\ &\leq 2 \bigg( \int_{\mathbb{R}_+} |h'(x)|^2 e^{\beta x} dx + \frac{\beta^2}{4} \int_{\mathbb{R}_+} |h(x)|^2 e^{\beta x} dx \bigg)
\\ &\leq 2 \| h \|_{\beta}^2 + \frac{\beta^2}{2} \| h \|_{L_{\beta}^2} \leq \bigg( 2 + \frac{\beta^2 C_1^2}{2} \bigg) \| h \|_{\gamma}^2,
\end{align*}
which, together with Lemma~\ref{lemma-reflection}, concludes the proof.
\end{proof}

For $h \in L^1(\mathbb{R})$ the \emph{Fourier transform} $\mathcal{F}h : \mathbb{R} \rightarrow \mathbb{C}$ is defined as
\begin{align}\label{Fourier-transform}
(\mathcal{F} h)(\xi) := \frac{1}{\sqrt{2 \pi}} \int_{\mathbb{R}} h(x) e^{-i \xi x} dx, \quad \xi \in \mathbb{R}.
\end{align}
Recall that $C_0(\mathbb{R})$ denotes the space of all continuous functions vanishing at infinity, which, equipped with the supremum norm, is a Banach space. We have the following result:

\begin{lemma}\label{lemma-Fourier-cont}
The Fourier transform $\mathcal{F} : L^1(\mathbb{R}) \rightarrow C_0(\mathbb{R})$ is a continuous linear operator with $\| \mathcal{F} \| \leq 1 / \sqrt{2 \pi}$.
\end{lemma}

\begin{lemma}\label{lemma-functional}
Let $\gamma > \beta > 0$ be arbitrary. Then the following statements are true:
\begin{enumerate}
\item For each $h \in H_{\gamma}^0$ we have $(h e^{(\beta / 2) \bullet}|_{(0,\infty)})^* \in L^1(\mathbb{R})$ and there is a constant $C_3 = C_3(\beta,\gamma) > 0$ such that
\begin{align*}
\| (h e^{(\beta / 2) \bullet}|_{(0,\infty)})^* \|_{L^1(\mathbb{R})} \leq C_3 \| h \|_{\gamma} \quad \text{for all $h \in H_{\gamma}^0$.}
\end{align*}

\item For each $\xi \in \mathbb{R}$ the mapping
\begin{align*}
H_{\gamma}^0 \rightarrow \mathbb{R}, \quad h \mapsto \mathcal{F}(h e^{(\beta / 2) \bullet}|_{(0,\infty)})^*(\xi)
\end{align*}
is a continuous linear functional.
\end{enumerate}
\end{lemma}

\begin{proof}
We set $\delta := \frac{1}{2}(\beta + \gamma) \in (\beta,\gamma)$.
Let $h \in H_{\gamma}^0$ be arbitrary. By the Cauchy-Schwarz inequality and Lemma~\ref{lemma-forward-Sobolev} we have
\begin{align*}
&\| (h e^{(\beta / 2) \bullet}|_{(0,\infty)})^* \|_{L^1(\mathbb{R})} = 2 \| h e^{(\beta / 2) \bullet} \|_{L^1(\mathbb{R}_+)} = 2 \int_{\mathbb{R}_+} |h(x) e^{(\beta / 2)x}| dx 
\\ &= 2 \int_{\mathbb{R}_+} |h(x) | e^{(\delta / 2)x} e^{-((\delta - \beta)/2) x} dx 
\\ &\leq 2 \bigg( \int_{\mathbb{R}_+} |h(x)|^2 e^{\delta x} dx \bigg)^{1/2} \bigg( \int_{\mathbb{R}_+} e^{-(\delta - \beta)x} dx \bigg)^{1/2} 
\\ &= 2 \sqrt{\frac{1}{\delta - \beta}} \| h \|_{L_{\delta}^2} \leq 2 C_1(\delta,\gamma) \sqrt{\frac{1}{\delta - \beta}} \| h \|_{\gamma},
\end{align*}
showing the first statement. Moreover, we have
\begin{align*}
\| e^{((\beta / 2) - \delta) \bullet} \|_{L_{\delta}^2}^2 = \int_{\mathbb{R}_+} e^{2((\beta / 2) - \delta)x} e^{\delta x} dx = \int_{\mathbb{R}_+} e^{-(\delta - \beta)x} dx = \frac{1}{\delta - \beta},
\end{align*}
showing that $e^{((\beta / 2) - \delta) \bullet} \in L_{\delta}^2$. Let $h \in H_{\gamma}^0$ and $\xi \in \mathbb{R}$ be arbitrary. By Lemma~\ref{lemma-forward-Sobolev} we have $h \in L_{\delta}^2$, and hence
\begin{align*}
&\mathcal{F}(h e^{(\beta / 2) \bullet}|_{(0,\infty)})^*(\xi) 
\\ &= \frac{1}{\sqrt{2 \pi}} \bigg( \int_0^{\infty} h(x) e^{(\beta/2) x} e^{- i \xi x} dx + \int_{-\infty}^0 h(-x) e^{-(\beta / 2) x} e^{- i \xi x} dx \bigg) 
\\ &= \frac{1}{\sqrt{2 \pi}} \bigg( \int_0^{\infty} h(x) e^{(\beta/2) x} e^{- i \xi x} dx + \int_0^{\infty} h(x) e^{(\beta / 2) x} e^{i \xi x} dx \bigg) 
\\ &= \frac{1}{\sqrt{2 \pi}} \big\langle h, e^{((\beta / 2) - \delta) \bullet} \big( e^{-i \xi \bullet} + e^{i \xi \bullet} \big) \big\rangle_{L_{\delta}^2},
\end{align*}
proving the second statement.
\end{proof}

We can also define the Fourier transform on $L^2(\mathbb{R})$ such that
$\mathcal{F} : L^2(\mathbb{R}) \rightarrow L^2(\mathbb{R})$ is a bijection and we have the Plancherel isometry
\begin{align}\label{Plancherel}
\langle \mathcal{F}f,\mathcal{F}g \rangle_{L^2(\mathbb{R})} = \langle f,g \rangle_{L^2(\mathbb{R})} \quad \text{for all $f,g \in L^2(\mathbb{R})$.}
\end{align}
Moreover, the two just reviewed definitions of the Fourier transform coincide on $L^1(\mathbb{R}) \cap L^2(\mathbb{R})$. For each $h \in W^1(\mathbb{R})$ we have
\begin{align}\label{Fourier-diff}
(\mathcal{F} h')(\xi) = i \xi (\mathcal{F} h)(\xi), \quad \xi \in \mathbb{R}.
\end{align}

\begin{lemma}\label{lemma-est-Fourier-diff}
For every $h \in W^1(\mathbb{R})$ we have
\begin{align*}
\| \bullet \mathcal{F} h \|_{L^2(\mathbb{R})} \leq \| h \|_{W^1(\mathbb{R})}.
\end{align*}
\end{lemma}

\begin{proof}
Let $h \in W^1(\mathbb{R})$ be arbitrary. By identity (\ref{Fourier-diff}) and the Plancherel isometry (\ref{Plancherel}) we have
\begin{align*}
\| \bullet \mathcal{F} h \|_{L^2(\mathbb{R})} = \| \mathcal{F} h' \|_{L^2(\mathbb{R})} = \| h' \|_{L^2(\mathbb{R})} \leq \| h \|_{W^1(\mathbb{R})},
\end{align*}
finishing the proof.
\end{proof}

\section{The embedding result and its proof}\label{sec-main}

In this section, we present the compact embedding result and its proof.

\begin{theorem}\label{thm-comp-embedding}
For all $\gamma > \beta > 0$ we have the compact embedding 
\begin{align*}
H_{\gamma} \subset \subset L_{\beta}^2 \oplus \mathbb{R}.
\end{align*}
\end{theorem}

\begin{proof}
Noting that $H_{\gamma} \cong H_{\gamma}^0 \oplus \mathbb{R}$, it suffices to prove the compact embedding $H_{\gamma}^0 \subset \subset L_{\beta}^2$.
Let $(h_j)_{j \in \mathbb{N}} \subset H_{\gamma}^0$ be a bounded sequence. Then there exists a subsequence which converges weakly in $H_{\gamma}^0$. Without loss of generality, we may assume that the original sequence $(h_j)_{j \in \mathbb{N}}$ converges weakly in $H_{\gamma}^0$. We shall prove that $(h_j)_{j \in \mathbb{N}}$ is a Cauchy sequence in $L_{\beta}^2$. According to Lemma~\ref{lemma-forward-Sobolev}, the sequence $(g_j)_{j \in \mathbb{N}}$ given by
\begin{align*}
g_j := ( h_j e^{(\beta / 2) \bullet}|_{(0,\infty)} )^*, \quad j \in \mathbb{N}
\end{align*}
is a bounded sequence in $W^1(\mathbb{R})$.
By Lemma~\ref{lemma-reflection} and the Plancherel isometry (\ref{Plancherel}), for all $j,k \in \mathbb{N}$ we get
\begin{align*}
&\| h_k - h_j \|_{L_{\beta}^2}^2 = \| h_k e^{(\beta / 2) \bullet} - h_j e^{(\beta / 2) \bullet} \|_{L^2(\mathbb{R}_+)}^2 \leq \| g_k - g_j \|_{L^2(\mathbb{R})}^2
\\ &= \| \mathcal{F} g_k - \mathcal{F} g_j \|_{L^2(\mathbb{R})}^2 
= \int_{\mathbb{R}} | (\mathcal{F} g_k)(x) - (\mathcal{F} g_j)(x) |^2 dx. 
\end{align*}
Thus, for every $R > 0$ we obtain the estimate
\begin{equation}\label{comp-conv}
\begin{aligned}
\| h_k - h_j \|_{L_{\beta}^2}^2 &\leq \int_{\{ |x| \leq R \}} | (\mathcal{F} g_k)(x) - \mathcal{F} (g_j)(x) |^2 dx 
\\ &\quad + \int_{\{ |x| > R \}} | (\mathcal{F} g_k)(x) - \mathcal{F} (g_j)(x) |^2 dx.
\end{aligned}
\end{equation}
By Lemma~\ref{lemma-est-Fourier-diff},  the sequence $( \bullet \mathcal{F} g_j )_{j \in \mathbb{N}}$ is bounded in $L^2(\mathbb{R})$. Therefore, for an arbitrary $\epsilon > 0$ there exists a real number $R > 0$ such that
\begin{equation}\label{comp-conv-1}
\begin{aligned}
&\int_{\{ |x| > R \}} | ( \mathcal{F} g_k)(x) - (\mathcal{F} g_j)(x) |^2 dx 
\\ &\leq \frac{1}{R^2} \int_{\{ |x| > R \}} |x|^2 | (\mathcal{F} g_k)(x) - (\mathcal{F} g_j)(x) |^2 dx < \epsilon \quad \text{for all $j,k \in \mathbb{N}$.}
\end{aligned}
\end{equation}
By Lemma~\ref{lemma-functional}, for each $\xi \in \mathbb{R}$ the mapping
\begin{align*}
H_{\gamma}^0 \rightarrow \mathbb{R}, \quad h \mapsto \mathcal{F}(h e^{(\beta / 2) \bullet}|_{(0,\infty)})^*(\xi)
\end{align*}
is a continuous linear functional. Consequently, since $(h_j)_{j \in \mathbb{N}}$ converges weakly in $H_{\gamma}^0$, for each $\xi \in \mathbb{R}$ the real-valued sequence $( (\mathcal{F} g_j)(\xi) )_{j \in \mathbb{N}}$
is convergent. Moreover, by Lemmas~\ref{lemma-Fourier-cont} and \ref{lemma-functional}, for all $h \in H_{\gamma}^0$ we have the estimate
\begin{align*}
\| \mathcal{F}((h e^{(\beta / 2) \bullet}|_{(0,\infty)})^*) \|_{C_0(\mathbb{R})} \leq \frac{1}{\sqrt{2 \pi}} \| (h e^{(\beta / 2) \bullet}|_{(0,\infty)})^* \|_{L^1(\mathbb{R})} 
\leq \frac{C_3}{\sqrt{2 \pi}} \| h \|_{\gamma}.
\end{align*}
Therefore, the sequence $( \mathcal{F} g_j )_{j \in \mathbb{N}}$ is bounded in $C_0(\mathbb{R})$. Using Lebesgue's dominated convergence theorem, we deduce that
\begin{align}\label{comp-conv-2}
\int_{\{ |x| \leq R \}} | (\mathcal{F} g_k)(x) - (\mathcal{F} g_j)(x) |^2 dx \rightarrow 0 \quad \text{for $j,k \rightarrow \infty$.}
\end{align}
Combining (\ref{comp-conv}) together with (\ref{comp-conv-1}) and (\ref{comp-conv-2}) shows that $(h_j)_{j \in \mathbb{N}}$ is a Cauchy sequence in $L_{\beta}^2$, completing the proof.
\end{proof}

\begin{remark}
Note that the proof of Theorem~\ref{thm-comp-embedding} has certain analogies to the proof of the classical Rellich embedding theorem (see, e.g., \cite[Satz~V.2.13]{Werner}), which states the compact embedding $H_0^1(\Omega) \subset \subset L^2(\Omega)$ for an open, bounded subset $\Omega \subset \mathbb{R}^n$. Here $H_0^1(\Omega)$ denotes the Sobolev space $H_0^1(\Omega) = \overline{\mathcal{D}(\Omega)}$, where $\mathcal{D}(\Omega)$ is the space of all $C^{\infty}$-functions on $\Omega$ with compact support, and where the closure is taken with respect to the topology induced by the inner product $\langle \cdot,\cdot \rangle_{W^1}$. Let us briefly describe the analogies and differences between the two results:
\begin{itemize}
\item In the classical Rellich embedding theorem the domain $\Omega$ is assumed to be bounded, whereas in Theorem~\ref{thm-comp-embedding} we have $\Omega = \mathbb{R}_+$. Moreover, we consider weighted function spaces with weight functions of the type $w(x) = e^{\beta x}$ for some constant $\beta > 0$. This requires a careful analysis of the results regarding Fourier transforms which we have adapted to the present situation; see Lemma~\ref{lemma-functional}.

\item $H_{\gamma}$ and $H_0^1(\Omega)$ are different kinds of spaces. While the norm on $H_0^1(\Omega)$ given by (\ref{norm-Sobolev}) involves the $L^2$-norms of a function $h$ and its derivative $h'$, the norm (\ref{norm-subspace}) on $H_{\gamma}$ only involves the $L^2$-norm of the derivative $h'$ and a point evaluation. Therefore, the embedding $H_0^1(\Omega) \subset L^2(\Omega)$ follows right away, whereas we require the assumption $\beta < \gamma$ for the embedding $H_{\gamma}^0 \subset L_{\beta}^2$; see Lemma~\ref{lemma-forward-Sobolev}.

\item The classical Rellich embedding theorem does not need to be true with $H_0^1(\Omega)$ being replaced by $W^1(\Omega)$. The reason behind it is that, in general, it is not possible to extend a function $h \in W^1(\Omega)$ to a function $\tilde{h} \in W^1(\mathbb{R}^n)$, which, however, is crucial in order to apply the results about Fourier transforms. Usually, one assumes that $\Omega$ satisfies a so-called cone condition, see, e.g., \cite{Adams} for further details. In our situation, we have to ensure that every function $h \in H_{\gamma}^0$ can be extended to a function $\tilde{h} \in W^1(\mathbb{R})$, and this is provided by Lemma~\ref{lemma-forward-Sobolev}.
\end{itemize}
\end{remark}

For the rest of this section, we shall describe the announced application regarding the approximation of solutions to semilinear stochastic partial differential equations (SPDEs), which in particular applies to the modeling of interest rates. Consider a SPDE of the form
\begin{align}\label{SPDE}
\left\{
\begin{array}{rcl}
dr_t & = & (A r_t + \alpha(t,r_t)) dt + \sigma(t,r_t)dW_t 
\\ && + \int_E \gamma(t,r_{t-},\xi) (\mathfrak{p}(dt,d\xi) - \nu(d\xi)dt) \medskip
\\ r_0 & = & h_0
\end{array}
\right.
\end{align}
on some separable Hilbert space $H_1$ with $A$ denoting the generator of some strongly continuous semigroup on $H_1$, driven by a Wiener process $W$ and a homogeneous Poisson random measure $\mathfrak{p}$ with compensator $dt \otimes \nu(d\xi)$ on some mark space $E$. 
We assume that the standard Lipschitz and linear growth conditions are satisfied which ensure for each initial condition $h_0 \in H_1$ the existence of a unique weak solution $r$ to (\ref{SPDE}), that is, for each $\zeta \in \mathcal{D}(A^*)$ we have almost surely
\begin{align*}
\langle \zeta,r_t \rangle &= \langle \zeta,h_0 \rangle_{H_1} + \int_0^t \big( \langle A^* \zeta,r_s \rangle_{H_1} + \langle \zeta,\alpha(s,r_s) \rangle_{H_1} \big) ds + \int_0^t \langle \zeta,\sigma(s,r_s) \rangle_{H_1} dW_s
\\ &\quad + \int_0^t \int_E \langle \zeta,\gamma(s,r_{s-},\xi) \rangle_{H_1} (\mathfrak{p}(ds,d\xi) - \nu(d\xi)ds) \quad \text{for all $t \geq 0$,}
\end{align*}
see, e.g., \cite{SPDE} for further details. Let $H_2$ be a larger separable Hilbert space with compact embedding $H_1 \subset \subset H_2$. By virtue of Theorem~\ref{thm-comp-embedding}, this is in particular satisfied for the forward curve spaces $H_1 = H_{\gamma}$ and $H_2 = L_{\beta}^2 \oplus \mathbb{R}$ for $\gamma > \beta > 0$. If, furthermore, $A = d/dx$ is the differential operator, which is generated by the translation semigroup $(S_t)_{t \geq 0}$ given by $S_t h = h(t + \bullet)$, and $\alpha = \alpha_{\rm HJM}$ is given by the so-called HJM drift condition
\begin{align*}
\alpha_{\rm HJM}(t,h) &= \sum_j \sigma^j(t,h) \int_0^{\bullet} \sigma^j(t,h)(\eta) d\eta 
\\ &\quad - \int_E \gamma(t,h,\xi) \bigg[ \exp \bigg( - \int_0^{\bullet} \gamma(t,h,\xi)(\eta)d\eta \bigg) - 1 \bigg] \nu(d\xi),
\end{align*}
then the SPDE (\ref{SPDE}), which in this case becomes the mentioned HJMM equation, describes the evolution of interest rates in an arbitrage free bond market; we refer to \cite{Positivity} for further details.

By virtue of the compact embedding $H_1 \subset \subset H_2$, there exist orthonormal systems $(e_k)_{k \in \mathbb{N}}$ of $H_1$ and $(f_k)_{k \in \mathbb{N}}$ of $H_2$, and a decreasing sequence $(s_k)_{k \in \mathbb{N}} \subset \mathbb{R}_+$ with $s_k \rightarrow 0$ such that
\begin{align*}
h = \sum_{k=1}^{\infty} s_k \langle h,e_k \rangle_{H_1} f_k \quad \text{for all $h \in H_1$,}
\end{align*}
see, e.g., \cite[Satz~VI.3.6]{Werner}. The numbers $s_k$ are the singular numbers of the identity operator ${\rm Id} : H_1 \rightarrow H_2$. Defining the sequence $(T_n)_{n \in \mathbb{N}}$ of finite-rank operators
\begin{align*}
T_n : H_1 \rightarrow F_n, \quad T_n h := \sum_{k=1}^n s_k \langle h,e_k \rangle_{H_1} f_k,
\end{align*}
where $F_n :=  \langle f_1,\ldots,f_n \rangle$,
we even have $T_n \rightarrow {\rm Id}$ with respect to the operator norm
\begin{align*}
\| T \| := \sup_{\| h \|_{H_1} \leq 1} \| Th \|_{H_2},
\end{align*}
see, e.g., \cite[Korollar~VI.3.7]{Werner}. Consequently, denoting by $r$ the weak solution to the SPDE (\ref{SPDE}) for some initial condition $h_0 \in H_1$, the sequence $(T_n(r))_{n \in \mathbb{N}}$ is a sequence of $F_n$-valued stochastic processes, and we have almost surely
\begin{align}\label{norm-conv}
\| T_n(r_t) - r_t \|_{H_2} \leq \| T_n - {\rm Id} \| \, \| r_t \|_{H_1} \rightarrow 0 \quad \text{for all $t \geq 0$,}
\end{align}
showing that the weak solution $r$ -- when considered on the larger state space $H_2$ -- can be approximated by the sequence of finite dimensional processes $(T_n(r))_{n \in \mathbb{N}}$ with distance between $T_n(r)$ and $r$ estimated in terms of the operator norm $\| T_n - {\rm Id} \|$, as shown in (\ref{norm-conv}).
However, the sequence $(T_n(r))_{n \in \mathbb{N}}$ does not need to be a sequence of It\^{o} processes. This issue is addressed by the following result:

\begin{proposition}\label{prop-approx}
Let $(\epsilon_n)_{n \in \mathbb{N}} \subset (0,\infty)$ be an arbitrary decreasing sequence with $\epsilon_n \rightarrow 0$. Then, for every initial condition $h_0 \in H_1$ there exists a sequence $(r^{(n)})_{n \in \mathbb{N}}$ of $F_n$-valued It\^{o} processes such that almost surely
\begin{align}\label{est-epsilon-n}
\| r_t^{(n)} - r_t \|_{H_2} \leq ( \| T_n - {\rm Id} \| + \epsilon_n ) \| r_t \|_{H_1} \rightarrow 0 \quad \text{for all $t \geq 0$,}
\end{align}
where $r$ denotes the weak solution to (\ref{SPDE}).
\end{proposition}

\begin{proof}
According to \cite[Theorems~13.35.c and 13.12]{Rudin}, the domain $\mathcal{D}(A^*)$ is dense in $H_1$. Therefore, for each $n \in \mathbb{N}$ there exist elements $\zeta_1^{(n)},\ldots,\zeta_n^{(n)} \in \mathcal{D}(A^*)$ such that
\begin{align*}
\| \zeta_k^{(n)} - e_k \|_{H_1} < \frac{\epsilon_n}{2^k \cdot s_k} \quad \text{for all $k=1,\ldots,n$,}
\end{align*}
where we use the convention $\frac{x}{0} := \infty$ for $x > 0$.
We define the sequence $(S_n)_{n \in \mathbb{N}}$ of finite-rank operators as
\begin{align*}
S_n : H_1 \rightarrow F_n, \quad S_n h := \sum_{k=1}^n s_k \langle h,\zeta_k^{(n)} \rangle_{H_1} f_k.
\end{align*}
By the geometric series, for all $n \in \mathbb{N}$ we have
\begin{align*}
\| S_n - {\rm Id} \| &\leq \| S_n - T_n \| + \| T_n - {\rm Id} \| \leq \sum_{k=1}^n s_k \| \langle \bullet,\zeta_k^{(n)} - e_k \rangle_{H_1} \| + \| T_n - {\rm Id} \| 
\\ &\leq \epsilon_n \sum_{k=1}^n \frac{1}{2^k} + \| T_n - {\rm Id} \| \leq \epsilon_n + \| T_n - {\rm Id} \|.
\end{align*}
For each $n \in \mathbb{N}$ let $r^{(n)}$ be the $F_n$-valued It\^{o} process
\begin{align*}
r_t^{(n)} = h_0^{(n)} + \int_0^t \alpha_s^{(n)} ds + \int_0^t \sigma_s^{(n)} dW_s + \int_0^t \int_E \delta_s^{(n)}(\xi) (\mathfrak{p}(ds,d\xi) - \nu(d\xi,ds)),
\end{align*}
with parameters given by
\begin{align*}
h_0^{(n)} &= \sum_{k=1}^n s_k \langle \zeta_k^{(n)},h_0 \rangle_{H_1} f_k, \quad \alpha_t^{(n)} = \sum_{k=1}^n s_k \big( \langle A^* \zeta_k^{(n)},r_t \rangle_{H_1} + \langle \zeta_k^{(n)},\alpha(t,r_t) \rangle_{H_1} \big) f_k,
\\ \sigma_t^{(n)} &= \sum_{k=1}^n s_k \langle \zeta_k^{(n)},\sigma(t,r_t) \rangle_{H_1} f_k, \quad \delta_t^{(n)}(\xi) = \sum_{k=1}^n s_k \langle \zeta_k^{(n)},\delta(t,r_{t-},\xi) \rangle_{H_1} f_k.
\end{align*}
Since $r$ is a weak solution to (\ref{SPDE}), we obtain almost surely
\begin{align*}
S_n(r_t) &= \sum_{k=1}^n s_k \langle \zeta_k^{(n)},r_t \rangle_{H_1} f_k
\\ &= \sum_{k=1}^n s_k \bigg( \langle \zeta_k^{(n)},h_0 \rangle_{H_1} + \int_0^t \big( \langle A^* \zeta_k^{(n)},r_s \rangle_{H_1} + \langle \zeta_k^{(n)},\alpha(s,r_s) \rangle_{H_1} \big) ds 
\\ &\qquad\qquad\,\,\, + \int_0^t \langle \zeta_k^{(n)},\sigma(s,r_s) \rangle_{H_1} dW_s 
\\ &\qquad\qquad\,\,\, + \int_0^t \int_E \langle \zeta_k^{(n)},\delta(s,r_{s-},\xi) \rangle_{H_1} (\mathfrak{p}(ds,d\xi) - \nu(d\xi,ds)) \bigg) f_k
\\ &= h_0^{(n)} + \int_0^t \alpha_s^{(n)} ds + \int_0^t \sigma_s^{(n)} dW_s + \int_0^t \int_E \delta_s^{(n)}(\xi) (\mathfrak{p}(ds,d\xi) - \nu(d\xi,ds)) 
\\ &= r_t^{(n)} \quad \text{for all $t \geq 0$,}
\end{align*}
which finishes the proof.
\end{proof}

We shall conclude this section with further consequences regarding the speed of convergence of the approximations $(r^{(n)})_{n \in \mathbb{N}}$ provided by Proposition~\ref{prop-approx}. Let $h_0 \in H_1$ be an arbitrary initial condition and denote by $r$ the weak solution to (\ref{SPDE}). Furthermore, let $T > 0$ be a finite time horizon. Since
\begin{align*}
\mathbb{E} \bigg[ \sup_{t \in [0,T]} \| r_t \|_{H_1}^2 \bigg] < \infty,
\end{align*}
see, e.g., \cite[Corollary~10.3]{SPDE}, by (\ref{est-epsilon-n}) there exists a constant $K > 0$ such that
\begin{align*}
\mathbb{E} \bigg[ \sup_{t \in [0,T]} \| r_t^{(n)} - r_t \|_{H_2}^2 \bigg]^{1/2} \leq K (\| T_n - {\rm Id} \| + \epsilon_n) \rightarrow 0,
\end{align*}
providing a uniform estimate for the distance of $r^{(n)}$ and $r$ in the mean-square sense. Moreover, considering the pure diffusion case
\begin{align*}
\left\{
\begin{array}{rcl}
dr_t & = & (A r_t + \alpha(t,r_t)) dt + \sigma(t,r_t)dW_t \medskip
\\ r_0 & = & h_0,
\end{array}
\right.
\end{align*}
the sample paths of $r$ are continuous, for every constant $K > \| h_0 \|_{H_1}$ the stopping time
\begin{align*}
\tau := \inf \{ t \geq 0 : \| r_t \| \geq K \}
\end{align*}
is strictly positive, and by (\ref{est-epsilon-n}) for the stopped processes we obtain almost surely
\begin{align}\label{uni-local}
\sup_{t \in \mathbb{R}_+} \| r_{t \wedge \tau}^{(n)} - r_{t \wedge \tau} \|_{H_2} \leq K (\| T_n - {\rm Id} \| + \epsilon_n) \rightarrow 0,
\end{align}
i.e., locally the solution $r$ stays in a bounded subset of $H_{\gamma}$ and we obtain the uniform convergence (\ref{uni-local}).

\section*{Acknowledgement}

The author is grateful to Vidyadhar Mandrekar for posing the question treated in this paper. The author also would like to thank an anonymous referee for valuable comments and suggestions.


\begin{thebibliography}{20}
\bibitem{Adams} Adams, R.~A., Fournier, J.~J.~F. (2003):
    \textit{Sobolev spaces.} Second Edition. Academic Press, Amsterdam.

\bibitem{Barski}
  Barski, M., Zabczyk, J. (2012):
  Heath-Jarrow-Morton-Musiela equation with L\'{e}vy perturbation.
  \textit{Journal of Differential Equations} {\bf 253}(9), 2657--2697.

\bibitem{Brezis} Brezis, H. (2011):
  \textit{Functional analysis, Sobolev spaces and partial differential equations.} Springer, New York.

\bibitem{fillnm} Filipovi\'c, D. (2001):
  \textit{Consistency problems for Heath--Jarrow--Morton interest rate models.} Springer, Berlin.

\bibitem{Filipovic-Tappe} Filipovi\'c, D., Tappe, S. (2008):
  Existence of L\'evy term structure models.
  \textit{Finance and Stochastics} {\bf 12}(1), 83--115.

\bibitem{SPDE} Filipovi\'c, D., Tappe, S., Teichmann, J. (2010):
  Jump-diffusions in Hilbert spaces: Existence, stability and numerics.
  \textit{Stochastics} {\bf 82}(5), 475--520.

\bibitem{Positivity} Filipovi\'c, D., Tappe, S., Teichmann, J. (2010):
  Term structure models driven by Wiener processes and Poisson measures: Existence and positivity.
  \textit{SIAM Journal on Financial Mathematics} {\bf 1}(1), 523--554.

\bibitem{Rudin} Rudin, W. (1991): \textit{Functional Analysis}.
Second Edition, McGraw-Hill.

\bibitem{Rusinek} Rusinek, A. (2010):
  Mean reversion for HJMM forward rate models \emph{Advances in Applied Probability} {\bf 42}(2), 371--391.

\bibitem{Werner} Werner, D. (2007):
    \textit{Funktionalanalysis.} Sixth Edition. Springer, Berlin.
\end{thebibliography}
\end{document}